\documentclass[12pt]{article}

\usepackage{amsfonts}
\usepackage[top=2.75cm, bottom=2.75cm, left=2.75cm, right=2.75cm]{geometry}
\usepackage{epsfig,amsmath,amssymb,amsthm,mathrsfs}
\usepackage[latin1]{inputenc}
\usepackage[english]{babel}
\usepackage[T1]{fontenc}
\usepackage{dcolumn}
\usepackage{bm}

\newtheorem{prop}{Proposition}
\newtheorem{cor}[prop]{Corollary}
\newtheorem{lem}[prop]{Lemma}

\newtheorem{dfn}[prop]{Definition}

\newtheorem{thm}[prop]{Theorem}

\newcommand{\fg}{\mbox{${\mathfrak g}$}}
\newcommand{\Gs}{\mbox{$G^{*}$}}

\newcommand{\fh}{\mbox{${\mathfrak h}$}}

\newcommand{\la}{\mbox{$\mathfrak g$}}

\newcommand{\cx}{\mbox{$X^{r}$}}
\newcommand{\onf}{\mbox{$1$-form}}

\newcommand{\ga}{\gamma}
\newcommand{\ul}{\underline}
\newcommand{\og}{{\overline g}}
\newcommand{\oga}{{\overline \gamma}}

\def\Q1{{\bf Q}_1}

\def\pg{p_{\mathfrak{g}}}
\def\pgs{p_{\mathfrak{g}^\star}}
\def\Ad{{\rm Ad}}

\def\ph{{p_H}}
\def\php{{p_{H^\perp}}}
\def\JH{{J_{\mbox{\tiny H}}}}
\def\h{{\mathfrak h}}

\def\sharpplus{{\pi_+^\sharp}}
\def\G{{\cal G}}

\def\pruno{{\rm pr_1}}
\def\prdue{{\rm pr_2}}
\def\alphaG{{\alpha_{\mbox{\tiny G}}}}
\def\betaG{{\beta_{\mbox{\tiny G}}}}
\def\alphaGs{{\alpha_{\mbox{\tiny G}^*}}}
\def\betaGs{{\beta_{\mbox{\tiny G}^*}}}
\def\alphaGsHp{{\alpha_{\mbox{\tiny G}^*}^{\mbox{\tiny H}^\perp}}}
\def\betaGsHp{{\beta_{\mbox{\tiny G}^*}^{\mbox{\tiny H}^\perp}}}
\def\alphaGH{{\alpha_{\mbox{\tiny G}}^{\mbox{\tiny H}}}}
\def\betaGH{{\beta_{\mbox{\tiny G}}^{\mbox{\tiny H}}}}
\def\mGsHp{{{\rm m}_{\mbox{\tiny G}^*}^{\mbox{\tiny H}^\perp}}}
\def\mGH{{{\rm m}_{\mbox{\tiny G}}^{\mbox{\tiny H}}}}
\def\prGuno{{{\rm pr}_{\G 1}}}
\def\prGdue{{{\rm pr}_{\G 2}}}
\def\prGHuno{{{\rm pr}_{\mbox{\tiny G}1}^{\mbox{\tiny H}}}}
\def\prGHdue{{{\rm pr}_{\mbox{\tiny G}2}^{\mbox{\tiny H}}}}
\def\tG{{\tilde G}}

\begin{document}
\thispagestyle{empty}

\smallskip
\begin{center} \LARGE
{\bf On the integration of \\ Poisson homogeneous spaces}
 \\[12mm] \normalsize
{\large\bf Francesco Bonechi$^{a}$, Nicola Ciccoli $^{b}$, Nicola Staffolani $^{c}$ and Marco Tarlini$^{a}$} \\[8mm]
 {\small\it
$^a$I.N.F.N. Sezione di Firenze,\\
  Via G. Sansone 1, 50019 Sesto Fiorentino - Firenze, Italy \\
  email: $\langle name \rangle$ @fi.infn.it
~\\
$^b$Dipartimento di Matematica e Informatica,\\
Via Vanvitelli 1, 06123 - Perugia, Italy\\
email: ciccoli@dipmat.unipg.it
~\\
$^c$Dipartimento di Fisica,\\
  Via G. Sansone 1, 50019 Sesto Fiorentino - Firenze, Italy \\
  email: staffolani@fi.infn.it
}
\end{center}
\vspace{10mm}
\centerline{\bfseries Abstract} We study a reduction procedure for describing the symplectic groupoid of a
Poisson homogeneous space obtained by quotient of a coisotropic subgroup. We perform it as a reduction of the Lu-Weinstein symplectic groupoid integrating Poisson Lie groups, that is suitable even for the non-complete case.  

\vspace{3.5cm}

\noindent {\bf Keywords:} Poisson geometry, Symplectic groupoids, Poisson homogeneous spaces, 
Poisson-Lie groups, Coisotropic subgroups, Geometric quantization.

\noindent {\bf MSC:}  53D05, 17B63, 22A22, 53D17, 53D50 

\noindent {\bf JGP subject code:} Symplectic geometry

\eject
\normalsize

\section{Introduction}

Symplectic groupoids were introduced by Karasev and Weinstein in the 80's, \cite{Ka,We1987} as a tool to quantize Poisson 
manifolds. They immediately became objects of independent math interests and one of the cornerstones of Poisson geometry. Our
knowledge on their role dramatically improved after the work by Cattaneo-Felder \cite{CaFe}, interpreting them as the phase space
of the Poisson sigma model, and Crainic-Fernandes \cite{CrFe} on the integrability of algebroids. On the contrary, quantum
aspects were much less studied: in \cite{XW} P.Xu and A.Weinstein defined the right notion of prequantization. Such a prequantization
can be explicitly constructed by reducing the prequantization of the phase space of the Poisson sigma model, as shown in \cite{BCZ}. In \cite{We2} the so-called noncommutative torus was recovered by the geometric quantization of the symplectic groupoid integrating the underling symplectic structure. Very recently, a notion of polarization of symplectic groupoids has been introduced in \cite{H}. Very roughly, the idea is to associate to any Poisson manifold a $C^*$-algebra constructed out of the space of polarized sections. This idea realizes a fundamental pattern in noncommutative geometry, where according to the Gelfand-Naimark theorem the noncommutative counterpart of the algebra of continuous functions on a (compact) manifold is a $C^*$-algebra.

During the same years, following the impulse of A. Connes, noncommutative geometry evolved trying to give  an axiomatic description of what a noncommutative manifold should be. The most studied examples can be collected in two main classes: $i$) the so-called $\theta$-manifolds, where the underlying Poisson structure is determined by the action of an abelian group, like the celebrated noncommutative torus, and the Connes-Landi $4$-sphere, see  \cite{CL} and also {\cite{CDV}}; $ii$) quantum groups and the associated homogeneous spaces. The $\theta$-manifolds are well studied from many different points of view: the symplectic groupoid was described first in \cite{Xu1}, their deformation quantization in \cite{Rieffel} and the polarization of the symplectic groupoid in \cite{H}. On the contrary, for spaces related to quantum groups less is known. The Poisson Lie groups, that are the semiclassical limit of quantum groups, were shown to be integrable in \cite{LW}. The problem of Poisson reduction of the symplectic groupoid was discussed in \cite{Xu}: from these results the integration of Poisson homogeneous spaces that are quotient by Poisson Lie subgroups can be obtained.  With a completely different approach, Poisson symmetric spaces (a definition which does not contain all covariant Poisson structures on symmetric spaces) were shown to be integrable in \cite{Xu2}. Nothing is known concerning the quantization of the symplectic groupoid.

Much information about quantum aspects comes from quantum group theory, especially for what concerns the study of $C^*$-algebras of basic examples of quantum groups and homogeneous spaces. It is an interesting program to look at these constructions from the point of view of the quantization of the symplectic groupoid like the one proposed in \cite{H}.

When we started this project, we realized that there was no construction of symplectic groupoids integrating the most important examples of Poisson homogeneous spaces coming from the semiclassical limit of homogeneous spaces of quantum groups, like Podles spheres, odd spheres, quantum grassmanians. The most relevant properties of the underlying Poisson structures is that  they can be obtained as quotients by coisotropic subgroups. The present paper is devoted to this construction and must be thought of as preliminary to the problem of quantization, that we hope to address in the future.
While we were finishing this paper, we were aware of \cite{Lu2007}, where a Poisson groupoid on any Poisson homogeneous space is presented. Moreover, conditions for the Poisson structure to be nondegenerate, so giving a symplectic groupoid, are discussed. The paper \cite{Lu2007} covers a large part of our results, that we obtained independently; in particular in the complete case it gives the symplectic groupoid that we describe in Theorem \ref{thm_groupoid}. Nevertheless, since the two approaches are different we think that our paper can help to clarify some issues. The differences between our paper and \cite{Lu2007} come out in the discussion of the noncomplete case ({\it i.e.} when dressing vector fields are not complete). In fact under a weaker hypothesis about the integrability of dressing vector fields (that we call {\it relative completeness}), we always obtain a symplectic groupoid. In concrete examples we show that the constructions are different.

This is the plan of the paper. In section 2 we recall very basic facts about Poisson manifolds and symplectic groupoids, mainly to fix notations. In section 3 we recall basic facts of Poisson Lie groups, following \cite{LuTh}. In section 4 we discuss the reduction procedure in terms of a moment map. When the Poisson Lie group is complete, the reduction is a straightforward analogue of the trivial case, where the action given by left multiplication of the subgroup can be lifted to a hamiltonian action on $T^*G$ and the groupoid is obtained by Marsden-Weinstein reduction. Indeed if the Poisson Lie group $G$ is complete then the left multiplication can be lifted to the groupoid and this action is hamiltonian in terms of Lu's momentum map. In the noncomplete case, the action can be lifted only as an infinitesimal action, and it is better to formulate it in terms of the symplectic action of the groupoid integrating the dual Poisson Lie group $G^*$. Nevertheless, if the subgroup satisfies the condition of relative completeness the procedure still works.

\bigskip
\bigskip

\section{Preliminaries on Poisson manifolds}

In this section we introduce the main definitions concerning the theory of Poisson manifolds and symplectic groupoids.

A {\it Poisson manifold} $P$ is a smooth manifold provided with a bivector $\pi_P\in \Gamma(\Lambda^2TP)$ satisfying $[\pi_P,\pi_P]_S=0$, where $[,]_S$ is the Schouten bracket between multivector fields. The cartesian product $P=M\times N$ of two Poisson manifolds is a Poisson manifold with Poisson tensor $\pi_{M\times N}=\pi_M+\pi_N$. The Poisson bivector $\pi_P$ defines a bundle map $\pi_P^\sharp:T^*P\rightarrow TP$ as $\langle \pi^\sharp_P(\omega_p),\nu_p\rangle = \langle \pi_P(p),\nu_p\wedge \omega_p\rangle$, for $p\in P$, $\omega_p,\nu_p\in T^*_pP$.
A submanifold $C$ of $P$ is a {\it coisotropic submanifold} if
$\pi^\sharp_P \left( N^* C \right) \subset TC$, where $N^* C$ is the conormal bundle of $C$, $N^*_x C = \{ \omega \in T^*_x P:~ \langle\omega , V\rangle = 0,\, \forall\ V \in T_x C\}$, for all $x \in C$. The generalized distribution defined by $\pi_P^\sharp(N^*C)$ is integrable and the space $\underline{C}$ of coisotropic leaves, if smooth, is a Poisson manifold. A submanifold $C$ is a {\it Poisson submanifold} if $\pi_P(c)\in\Lambda^2T_cC$; it is coisotropic and the coisotropic foliation is trivial. A smooth map $\Psi:M\rightarrow N$ between two Poisson manifolds is a {\it Poisson map} if the Poisson tensors are $\Psi$-related. In \cite{We1} it is proven that if $\Psi: M \rightarrow N$ is a Poisson map and $\mathcal{O}_N \subset N$ is a symplectic leaf then $\Psi^{-1} (\mathcal{O}_N) \subset M$, whenever is a submanifold, is coisotropic.

Let $\G=(\G,\G_0,\alpha_\G,\beta_\G,m_\G,\iota_\G,\epsilon_\G)$ be a Lie groupoid over the space of unities $\G_0$, where $\alpha_\G, \beta_\G:{\cal G}\rightarrow \G_0$ are the source and target maps, respectively, $m_\G:\G^{(2)}\rightarrow\G$ is the multiplication, $\iota_\G:\G\rightarrow\G^{-1}$ is the inversion and $\epsilon_\G:\G_0\rightarrow\G$ is the embedding of unities. Our conventions are that $(x_1,x_2)\in{\cal G}^{(2)}$ if $\beta_\G(x_1)=\alpha_\G(x_2)$. We say that $\G$ is source simply connected (ssc) if $\alpha_\G^{-1}(m)$ is connected and simply connected for any $m\in \G_0$.

A {\it symplectic groupoid} is a Lie groupoid, which is equipped with a symplectic form, such that the graph
of the multiplication is a lagrangian submanifold of ${\cal G} \times {\cal G} \times \bar{\cal G}$, where $\bar{\cal G}$ means $\cal G$ with the opposite symplectic structure.
There exists a unique Poisson structure on $\G_0$ such that $\alpha_\G$ and $\beta_\G$ are Poisson and anti-Poisson mappings, respectively. A Poisson manifold is said to be integrable if it is the space of units of a symplectic groupoid.

An equivalent characterization for a Lie groupoid $\cal G$ to be a symplectic groupoid is that the symplectic form $\omega$ of $\cal G$ be multiplicative, {\it i.e.} let $\prGuno$, $\prGdue: {\cal G}^{(2)}\rightarrow {\cal G}$ be respectively the projections onto the first and
second factor, then $m^*_\G \omega = \prGuno^* \omega + \prGdue^* \omega$.

Following \cite{MW}, we define the left action of $\cal G$ on a manifold $P$ with anchor $J:P\rightarrow \G_0$ a mapping from ${\cal G}{}_{\beta_\G}\times_J P=\{(x,p)\in\G\times P\,|\,\beta_\G(x)=J(p)\}$ to $P$, given by $(x,p)\rightarrow xp$ such that
{\it i}) $J(xp)=\alpha_\G(x)$, {\it ii}) $(xy)p =x(yp)$, {\it iii}) $\epsilon(J(p))p=p$. In the case of $P$ symplectic, the action of $\G$ is called {\it symplectic} if the graph of the action $\{(x,p,xp), \beta_\G(x)=J(p)\}$ is lagrangian in $\G\times P\times \bar P$. In \cite{MW} it is shown that $J:P\rightarrow \G_0$ is a Poisson map. Symplectic reduction is defined as follows. The isotropy group  $\G^m_m=\alpha^{-1}_\G(m)\cap\beta^{-1}_\G(m)$ of $m\in \G_0$ leaves invariant $J^{-1}(m)$ and $\G^m_m\backslash J^{-1}(m)$, whenever a manifold, is symplectic.

\bigskip\bigskip

\section{The symplectic groupoid of a Poisson Lie group}

In this section we recall basic results of the theory of Poisson Lie groups, and of the symplectic groupoid integrating them. The presentation follows \cite{LuTh}.

A {\it Poisson Lie group} $G$ is a Poisson manifold and a Lie group whose multiplication map $G \times G \rightarrow G$ is a Poisson map, where $G \times G$ is endowed with the product Poisson structure.
In terms of the Poisson bivector field $\pi_G$, it means that
\begin{equation}\label{multiplicativity}
 \pi_G (gh) = l_g \pi_G (h) + r_h \pi_G (g), \, \forall g,h \in G\, ,
\end{equation}
where $l_g$ ($r_g$) stands for the left (right) group multiplication by $g$, as well as for the induced tangent map. A multivector field satisfying (\ref{multiplicativity}) is said to be {\it multiplicative}.

A left action $\sigma : G \times P \rightarrow P$ of a Poisson Lie Group $G$ on a Poisson manifold $P$ is called a {\it Poisson action} if $\sigma$ is a Poisson map, where $G \times P$ is endowed with the product Poisson structure. In terms of the Poisson bivectors $\pi_G$ of $G$ and $\pi_P$ of $P$, $\sigma$ is a Poisson action if, for any $g\in G$ and $p\in P$, we have that
$$
\pi_P(\sigma(g,p))=g_*\pi_P(p) + p_*\pi_G(g)\;,
$$
where $g:P\rightarrow P$, $g(p)=\sigma(g,p)$ and $p:G\rightarrow P$, $p(g)=\sigma(g,p)$.

A {\it Lie bialgebra} is the couple $(\mathfrak{g},\mathfrak{g}^\star)$, where $\mathfrak{g}={\rm Lie} G$ and its dual $\mathfrak{g}^\star$ is a Lie algebra with bracket map $[~,~]_{\mathfrak{g}^\star}$ such that $\delta=[,]_{\mathfrak{g}^\star}^*:\mathfrak{g}\rightarrow\Lambda^2\mathfrak{g}$ is a $1$-cocycle on $\mathfrak{g}$ relative to the adjoint representation of $\fg$ on $\wedge^2 \mathfrak{g}$.

\smallskip
Let us assume that the group $G$ is connected and
simply connected.

\begin{thm}\label{bialgebrasversuspoissonlie}
There is a one to one correspondence between connected and simply connected Poisson Lie groups and Lie bialgebras.
\end{thm}

\smallskip
\begin{dfn}
The \textbf{double Lie algebra} $\mathfrak{d} = \fg \bowtie \fg^\star$ of the Lie bialgebra $(\fg, \fg^\star)$ is defined as the vector space $\fg \oplus \fg^\star$ endowed with the unique Lie bracket structure such that
\begin{itemize}
 \item[$\imath$)] it restricts to the given Lie brackets on $\fg$ and $\fg^\star$;
 \item [$\imath \imath$)] the symmetric and non-degenerate scalar product on $\fg \oplus \fg^\star$ defined by
\begin{equation}
 \langle X + \xi , Y + \eta\rangle= \xi(Y) + \eta(X),\, \forall X,Y \in \fg,\, \forall \xi, \eta \in \fg^\star \nonumber
\end{equation}
 is invariant.
\end{itemize}
\end{dfn}

In particular the bracket is defined for any $X,Y\in\mathfrak{g}$ and $\xi,\eta\in\mathfrak{g}^\star$ as
\begin{equation}
\label{bracket_double}
[X+\xi, Y+\eta] = [X,Y] -ad^*_\eta(X) +ad^*_\xi(Y) + [\xi,\eta] + ad^*_X(\eta) - ad^*_Y(\xi)\;.
\end{equation}

It can be shown that $\mathfrak{d}=\mathfrak{g}\oplus\mathfrak{g}^\star$ equipped with the bracket (\ref{bracket_double}) is a Lie algebra if and only if $(\mathfrak{g},\mathfrak{g}^\star)$ is a Lie bialgebra. As a consequence we have that if $(\fg, \fg^\star)$ is a Lie bialgebra, then so is $(\fg^\star ,\fg)$. In particular the connected and simply connected group $G^*$ integrating $\mathfrak{g}^\star$ is a Poisson-Lie group. We call it the {\it dual Poisson Lie group} of ($G,\pi_G$).

The {\it double Lie group} $D$ is defined as the connected and simply-connected Lie group with Lie algebra $\mathfrak{d}$. Let $\phi_{1}: G \rightarrow D$ and $\phi_{2}:
G^* \rightarrow D$ be the Lie group
homomorphisms obtained by respectively integrating the
inclusion maps ${\mathfrak g} \hookrightarrow {\mathfrak d}$ and ${\mathfrak g}^\star \hookrightarrow {\mathfrak d}$. In the following we denote $\phi_1(g)=\og$ and $\phi_2(\gamma)=\oga$.

The following formulas, proved in \cite{LuTh}, describe the Poisson tensor of $G$ and $G^*$ in terms of the group structure of $D$, making explicit the correspondence described in the integration Theorem \ref{bialgebrasversuspoissonlie}. Indeed, let $\pg:\mathfrak{d}\rightarrow\mathfrak{g}$, $\pgs:\mathfrak{d}\rightarrow\mathfrak{g}^\star$ be the natural projections; for any $g\in G$, $\gamma\in G^*$, $X_i\in\mathfrak{g}$ and $\xi_i\in\mathfrak{g}^\star$ we have that
\begin{eqnarray}
\label{poisson_tensors}
\langle r_{g^{-1}}\pi_G(g),\xi_1\wedge\xi_2\rangle &=& -\langle \pg\Ad_{\og^{-1}}\xi_1,\pgs\Ad_{\og^{-1}}\xi_2\rangle ~~~~ g\in G,~\xi_i\in\mathfrak{g}^\star\;,\cr
\langle r_{\gamma^{-1}}\pi_{G^*}(\gamma),X_1\wedge X_2\rangle &=&\langle \pg\Ad_{\oga^{-1}}X_1,\pgs\Ad_{\oga^{-1}}X_2\rangle\;~~~ \gamma\in G^*, X_i\in\mathfrak{g}\;,
\end{eqnarray}
where $\Ad$ is the adjoint action of $D$.

 For further purposes, let us consider the Poisson tensor $\pi_{+}$ on $D$ defined as follows:
\[ \pi_{+} (d) ~ = ~ {\frac{1}{2}} (r_{d} \pi_{0} ~ + ~ l_{d} \pi_{0}), ~~~~~~~~~d \in D, \]
where $\pi_{0} \in {\mathfrak d} \wedge {\mathfrak d}$ is defined by
 $\pi_{0} (\xi_{1} + X_{1}, ~ \xi_{2} + X_{2} )  =  \langle X_{1}, ~ \xi_{2}\rangle  -  \langle X_{2}, ~  \xi_{1}\rangle$,
for $\xi_{1} + X_{1}, \xi_{2} + X_{2} \in {\mathfrak d}^{*} \cong {\mathfrak g}^\star \oplus {\mathfrak g}$. If $d \in D$ can be factorized as $d = \og \oga$ for some $g \in G$ and $\ga \in G^{*}$, then an explicit formula for $\pi_{+}$ is given by
\begin{eqnarray}
& & \langle (l_{\og^{-1}}r_{\oga^{-1}})  \pi_{+}
(d), (\xi_{1} + X_{1})\wedge  (\xi_{2} + X_{2}) \rangle = \label{symplectic_tensor}\\
  & = & \langle X_{1}, \xi_{2}\rangle - \langle X_{2}, \xi_{1}\rangle + \langle l_{g^{-1}} \pi_G(g),\xi_{1}\wedge\xi_{2}\rangle
  + \langle r_{\ga^{-1}}\pi_{G^*} (\gamma), X_{1}\wedge
X_{2}\rangle = \nonumber \\
& = & \langle X_{1}, ~ \xi_{2} + Ad_{\oga} \pgs Ad_{\oga^{-1}} X_{2} \rangle
 ~ - ~ \langle\xi_{1}, ~ X_{2} + Ad_{\og^{-1}} \pg Ad_{\og} \xi_{2}\rangle \;.\nonumber
\end{eqnarray}
It can be proved that $\pi_+(g\gamma)$ is nondegenerate.

If $\phi_{1} \times \phi_{2}$ is a global diffeomorphism, then we can identify $D$ with $G\times G^*$ and $\pi_+$ defines a
symplectic structure on the double. Moreover the global decomposition of $D$ defines a left action of $G$ on $G^*$ and a
right action of $G^*$ on $G$. Let $g\in G$ and $\gamma\in G^*$ and let $g\gamma={}^g\gamma ~g^\gamma$, where we identify $g$ with
$\og$ and $\gamma$ with $\oga$. It is immediate to verify that $(g,\gamma)\rightarrow {}^g\gamma$ is a left action of $G$ on
$G^*$ and $(g,\gamma)\rightarrow g^\gamma$ is right action of $G^*$ on $G$. These are known as {\it dressing actions}. It can be
easily verified that for any $g,g_i\in G$ and $\gamma,\gamma_i\in G^*$ we have that
\begin{equation}\label{properties_dressing}
 (g_1 g_2)^\ga = g_1^{{}^{g_{2}}\ga} ~ g_2^\ga~~;~~~~~~~~~
 {}^g (\ga_1 \ga_2) = {}^g \ga_1 ~~ {}^{g^{\ga_{1}}}
 \ga_2~~.
\end{equation}
Such \emph{intertwining} property between the two actions defines
what is called a {\it matched pair of Lie groups} \cite{Mj,LW2}; we will come back to this notion in Section \ref{generalcase}.
\begin{lem} \label{lem_fund_vec}
The fundamental vector fields associated to the left dressing action of $G$ on $G^*$ and to the right dressing action of $G^*$ on $G$ are respectively:
\begin{equation}\label{dressing_vector_fields}
{\mathcal S}_{X} (\ga) = \pi_{G^*}^\sharp (r_{\ga^{-1}}^* X)~, ~~~~ \forall \ga \in G^*,~~ X \in \fg \equiv \left( \fg^\star \right)^\star~;
\end{equation}
\[
{\mathcal S}_{\xi} (g) = - \pi_G^\sharp (l_{g^{-1}}^* \xi) ~, ~~~~ \forall g \in G, ~~\xi \in \fg^\star~.
\]
\begin{proof}
A direct computation gives the following expressions for the fundamental vector fields associated with the left dressing action
of $G$ on $G^*$ and with the right dressing action of $G^*$ on $G$:
\[
 {\mathcal S}_{X} (\ga) = l_{\ga} ~ \pgs
 \left( \Ad_{\oga^{-1}} X \right), ~~~~ \forall \ga \in G^*, ~~ X \in \fg;
\]
\[
 {\mathcal S}_{\xi} (g) = r_g ~ \pg \left( \Ad_\og\xi
 \right), ~~~~ \forall g \in G, ~~ \xi \in \fg^\star.
\]
We have to prove that the pointwise pairing of these vector fields with generic $1$-forms coincide. Then, given $X,Y\in \fg$,$\ga \in G^*$,
\begin{eqnarray}
 \langle{\mathcal S}_{X} (\ga) , r_{\ga^{-1}}^* Y\rangle &=&
 \langle l_{\ga} \pgs \left( \Ad_{\oga^{-1}} X  \right) ,
 r_{\ga^{-1}}^* Y\rangle   \nonumber \\
 &=& \langle \pgs \left( \Ad_{\oga^{-1}} X \right) , \Ad_{\oga^{-1}}^* Y\rangle =
 \langle \pgs \Ad_{\oga^{-1}} X , \pg \Ad_{\oga^{-1}}
 Y\rangle  \nonumber \\
 &=& \langle\pi_{G^*} (\ga), r^*_{\ga^{-1}} (Y\wedge X)\rangle
 \equiv \langle\pi_{G^*}^\sharp (r_{\ga^{-1}}^* X),
 r_{\ga^{-1}}^* Y\rangle ~~\ .\nonumber
\end{eqnarray}
The proof for $\mathcal{S}_{\xi} (g)$ is similar.
\end{proof}
\end{lem}

The vector fields (\ref{dressing_vector_fields}) are called {\it dressing vector fields}; their definition depends only on the
infinitesimal Lie bialgebra. Therefore they are defined even when $\phi_1\times\phi_2$ is not a diffeomorphism (and more generally
even if $\phi_1$, $\phi_2$ does not exist). We saw that if $D=G\times G^*$, then the dressing vector fields are complete. In
\cite{LuTh}, Lu has proved that $i$) the dressing vector fields of $G$ are complete if and only if those of $G^*$ are complete; $ii$)
$D=G\times G^*$ if and only if the dressing vector fields are complete.

Integrability of Poisson Lie groups has been shown in \cite{LW}. Let us consider the submanifold of $G\times G^*\times G^*\times G$ of dimension $2\dim G$ defined by
\begin{equation}\label{sympl_grpd_noncomplete}
\Omega=\{(g_1,\gamma_1,\gamma_2,g_2)\in G\times G^*\times G^*\times G,\; \og_1\oga_1=\oga_2 \og_2\in D\}\;.   \end{equation}
The local diffeomorphism $\Phi:\Omega\rightarrow D$, defined as $\Phi(g_1,\gamma_1,\gamma_2,g_2)=\og_1\oga_1$, induces a nondegenerate Poisson structure on $\Omega$, that we still denote with $\pi_+$.

\begin{prop}\label{sympl_gpd} Let $G$ be a connected and simply connected Poisson Lie group and let $G^*$ be the dual Poisson Lie group. Consider the groupoid $\G(G)=(\Omega,\alpha_G,\beta_G,m_G,\epsilon_G,i_G)$ over $G$ with structure maps:
\begin{itemize}
 \item[i)] $\alphaG(g_1,\gamma_1,\gamma_2,g_2)=g_1$;
 \item[ii)] $\betaG(g_1,\gamma_1,\gamma_2,g_2)=g_2$;
 \item[iii)] $\epsilon_{\mbox{\tiny G}}(g)=(g,e,e,g)$;
 \item[iv)] $m_{\mbox{\tiny G}}[(g_1,\gamma_1,\gamma_2,g_2)(g_2,\lambda_1,\lambda_2,k_2)]=(g_1,\gamma_1\lambda_1,\gamma_2\lambda_2,k_2)$;
 \item[v)] $\iota_{\mbox{\tiny G}}(g_1,\gamma_1,\gamma_2,g_2)=(g_2,\gamma^{-1}_1,\gamma_2^{-1},g_1)$.
\end{itemize}
Then $\G(G)$ equipped with $\pi_+^{-1}$ is a symplectic groupoid integrating $(G,\pi_G)$.

Consider the groupoid 
$\G(G^*)=(\Omega,\alpha_{G^*},\beta_{G^*},m_{G^*},\epsilon_{G^*},i_{G^*})$ over $G^*$,
with structure maps:
\begin{itemize}
 \item[i)] $\alphaGs(g_1,\gamma_1,\gamma_2,g_2)=\gamma_2$;
 \item[ii)] $\betaGs(g_1,\gamma_1,\gamma_2,g_2)=\gamma_1$;
 \item[iii)] $\epsilon_{\mbox{\tiny G}^*}(\gamma)=(e,\gamma,\gamma,e)$;
 \item[iv)] $m_{\mbox{\tiny G}^*}[(g_1,\gamma_1,\gamma_2,g_2)(k_1,\lambda_1,\gamma_1,k_2)]=(g_1k_1,\lambda_1,\gamma_2,g_2k_2)$;
 \item[v)] $\iota_{\mbox{\tiny G}^*}(g_1,\gamma_1,\gamma_2,g_2)$ $=(g_1^{-1},\gamma_2,\gamma_1,g_2^{-1})$.
\end{itemize}
Then $\G(G^*)$ equipped with $-\pi_{+}^{-1}$ is a symplectic groupoid integrating $(G^*,\pi_{\mbox{\tiny G}^*})$.
\end{prop}

If $G$ and $G^*$ are complete, then $\Omega=G\times G^*$ globally and the above description can be given in terms of the dressing transformations. In particular the groupoid structures for $\G(G)$ read as $\alphaG(g\gamma)=g$, $\betaG(g\gamma)=g^\gamma$, $m_{\mbox{\tiny G}}[(g_1\gamma_1)(g_1^{\gamma_1}\gamma_2)]=(g_1\gamma_1\gamma_2)$, $\epsilon_{\mbox{\tiny G}}(g)=(ge)$, $\iota_{\mbox{\tiny G}}(g\gamma)=g^\gamma \gamma^{-1}$. For $\G(G^*)$ we have $\alphaGs(g\gamma)={}^g\gamma$, $\betaGs(g\gamma)=\gamma$, $m_{\mbox{\tiny G}^*}[(g_1{}^{g_2}\gamma_2)(g_2\gamma_2)]=(g_1g_2\gamma_2)$, $\epsilon_{\mbox{\tiny G}^*}(\gamma)=(e\gamma)$, $\iota_{\mbox{\tiny G}^*}(g\gamma)=g^{-1}{}^g\gamma$.

\medskip
\subsection{The non simply connected case} Let us remove in this subsection the hypothesis that $G$ is simply connected. The above construction of the symplectic groupoid cannot be repeated since now $\phi_1: \tG\rightarrow D$, where $\tG$ is the universal
covering of $G$. Let $Z\subset\tG$ be the discrete central subgroup such that $G=\tG/Z$. There exists on $\tG$ a unique multiplicative
Poisson structure $\pi_\tG$ such that the quotient $\tG\rightarrow G$ is a Poisson map and $\pi_\tG(z)=0$ for any $z\in Z$. As a
consequence the multiplication by $z$ on $\tG$ is a Poisson diffeomorphism; moreover by looking at (\ref{poisson_tensors}) we
see that since $\pi_\tG(z)=0$ we have $\Ad_{\overline z}\xi=\Ad^*_z\xi$, for any $\xi\in\mathfrak g^*$, and
$\Ad_z^*\xi=\xi$ since $Z$ is central. So we can conclude that $\phi_1(z)=\overline{z}$ commutes with $\oga$ for
any $\gamma\in G^*$ and that $Z$ acts as a symplectic groupoid morphism on the symplectic groupoid $\G(\tG)$ defined in
Proposition \ref{sympl_gpd} as $z(\tilde{g}_1,\gamma_1,\gamma_2,\tilde{g}_2)=(z\tilde{g}_1,\gamma_1,\gamma_2,z\tilde{g}_2)$.

\begin{prop}\label{sympl_grpd_non1conn}
For any Poisson Lie group $G=\tilde{G}/ Z$, $\G(G)=\G(\tG)/Z$ carries the structure of a symplectic groupoid integrating it.
\end{prop}

In the following we will denote the equivalence classes as $[\tilde{g}_1,\gamma_1,\gamma_2,\tilde{g}_2]\in \G(G)$. Remark
that it can happen that $\phi_1:\tilde{G}\rightarrow D$ satisfies $\phi_1(Z)=1$ so that it descends to $G$. In this case it is possible to define a groupoid as in Proposition \ref{sympl_gpd}, even without assuming that $G$ is simply connected. It is
easily observed that such groupoid is a quotient by $Z$ of the groupoid defined in Proposition \ref{sympl_grpd_non1conn}.

As a simple consequence of Proposition \ref{sympl_gpd}, we have the following corollary.

\begin{cor} \label{grpd_action} The symplectic groupoid $\G(G^*)$ acts simplectically on $\overline{\G(G)}$ with anchor $J:\G(G)\rightarrow G^*$ defined as $J[\tilde{g}_1,\gamma_1,\gamma_2,\tilde{g}_2]=\alphaGs(\tilde{g}_1,\gamma_1,\gamma_2,\tilde{g}_2)=\gamma_2$; the action  $a:\G(G^*){}_\betaGs\!\times_J\G(G)\rightarrow \G(G)$ is given by
$$
a\{(\tilde{k}_1,\lambda_1,\lambda_2,\tilde{k}_2)[\tilde{g}_1,\gamma_1,\lambda_1,\tilde{g}_2]\}=[\tilde{k}_1\tilde{g}_1,\gamma_1,\lambda_2,\tilde{k}_2\tilde{g}_2]\;.
$$
\end{cor}
\begin{proof}
Simply observe that the graph of the action $a$ is the quotient under the action of $Z$ of the graph of the multiplication of $\G(G^*)$.
\end{proof}

In particular we have that $J:(\G(G),\pi_+)\rightarrow (G^*,\pi_{\mbox{\tiny G}^*})$ is an anti-Poisson map.

\bigskip\bigskip

\section{The symplectic groupoid of a homogeneous space}

In this section we discuss the integration of Poisson homogeneous spaces of the Poisson Lie group $G$.

Let us start with the simplest case of a Lie group $G$ with the zero Poisson structure $\pi_G=0$. Its symplectic groupoid $\G(G)=G\times \mathfrak{g}^\star$ is identified with the cotangent bundle after trivializing via left
translations. The left multiplication of $G$ on itself admits a cotangent lift $k(g,\xi)=(kg,\xi)$ for $k,g\in G$ and $\xi\in
\mathfrak{g}^\star$; this lifted action is hamiltonian with momentum map $J(g,\xi)=\Ad^*_g(\xi)$. Let now $H$ be any closed
subgroup of $G$ with Lie algebra $\mathfrak h$. The restriction to $H$ of the lifted action is again obviously hamiltonian with
momentum map $\JH={\rm pr}_H\circ J$, where ${\rm pr}_H:{\mathfrak g}^\star\rightarrow {\mathfrak g}^\star/{\mathfrak h}^\perp$ and $\mathfrak{h}^\perp = \{ \xi \in g^\star:~ \langle\xi, X\rangle=0, \forall X \in \fh \}$. 
The symplectic groupoid of the quotient is just the Marsden--Weinstein reduction of this hamiltonian action: $T^*(H\backslash
G)=H\backslash \JH^{-1}(0)$.

We are going to see how this generalizes to a generic Poisson Lie group. We know that any action by Poisson diffeomorphism on an
integrable Poisson manifold can be lifted to a hamiltonian action on the (ssc) symplectic groupoid with a multiplicative momentum
map, see \cite{MW,FOR}. This construction applies only to the case $H$ being a Poisson subgroup with zero Poisson structure; it is
clear that in the general case one has to consider generalized notions of hamiltonian actions and even in the generalized setting
we will consider, the lifting property will not be automatic.

\subsection{Embeddable homogeneous spaces}
Let us recall some basic facts about coisotropic subgroups and their role in the quotient of Poisson manifolds. Let $H$ be a coisotropic connected closed subgroup of the Poisson Lie group $G$ and $\mathfrak{h} = Lie~H$. At the infinitesimal level, coisotropic subgroups are characterized by the following Proposition, whose proof can be found in \cite{STT}.

\begin{prop} \label{thm_coi_alg}
A subgroup $H$ is coisotropic if and only if $\mathfrak{h}^\perp \subset \fg^\star$ is a subalgebra of $\fg^\star$, where $\mathfrak{h}^\perp = \{ \xi \in g^\star:~ \langle\xi, X\rangle=0, \forall X \in \fh \}$.
\end{prop}

\noindent{\bf Assumption}: Let $\mathfrak{h}^\perp$ be integrated by a closed subgroup $H^\perp \subset G^*$.

\smallskip
As a consequence of Proposition \ref{thm_coi_alg}, $H^\perp$ results coisotropic as well. The following property of coisotropic subgroups will be relevant in what follows.

\begin{lem} \label{lem_h_hperp_inacca}
Given a coisotropic subgroup $H \subset G$, then the restriction of the (infinitesimal) dressing actions of $G$ on $G^*$ to $H$ leaves $H^\perp$ invariant and its orbits are the coisotropic leaves. Moreover, when $H$ is a Poisson subgroup, then the dressing action of $H^\perp$ on $H$ is trivial.
\end{lem}

\begin{proof}
Since for any $\gamma\in H^\perp$ we can characterize $T_\gamma H^\perp$ as $r_\gamma\h^\perp$, we get that $N^*_\gamma H^\perp=r^*_{\gamma^{-1}}\h$. Then the dressing vector fields $\mathcal{S}_{X} (\ga) = \pi_{\mbox{\tiny G}^*}^\sharp (r_{\ga^{-1}}^* X)$ corresponding to $X\in\h$ (see Lemma \ref{lem_fund_vec}) span the coisotropic distribution and, in particular, are tangent to $H^\perp$.
Analogously for the right dressing action. To prove the last statement, let us recall
that the coisotropic foliation of a Poisson submanifold is null. \end{proof}

\begin{cor}
If $H$ is a Poisson subgroup, then $H$ acts on $H^\perp$ by automorphisms.
\end{cor}

\begin{proof}
Let us prove it in the complete case. Let $h\in H$, $\gamma_i\in H^\perp$. By the properties (\ref{properties_dressing}) of the dressing action, ${}^h (\ga_1 \ga_2) = {}^h \ga_1 ~~ {}^{h^{\ga_1}} \ga_2$. As $h^{\ga_1} = h$ by means of Lemma \ref{lem_h_hperp_inacca}, then ${}^h (\ga_1 \ga_2) = {}^h \ga_1 ~~ {}^h \ga_2$. In the general case, infinitesimal action by automorphisms means that the dressing vector fields of $H$ on $H^\perp$ are multiplicative.
\end{proof}

 Let us denote with $\tilde{H}$, $\tilde{G}$ the universal covers of $H$ and $G$ respectively and let $\phi_{\tilde{H}}:\tilde{H}\rightarrow \tG$ the group homomorphism integrating the inclusion $\mathfrak{h}\rightarrow\mathfrak{g}$. Then let us define
\begin{equation}\label{subgroupoidofH}
 \Omega(\tilde{H},H^\perp) = \{(\phi_{\tilde H}(\tilde{h}_1),\gamma_1,\gamma_2,\phi_{\tilde H}(\tilde{h}_2))\in\G(\tG),\; \tilde{h}_i\in\tilde{H},\, \gamma_i\in H^\perp\}\subset \Omega\;.
\end{equation}
It is clear that $\Omega(\tilde{H},H^\perp)$ defines a subgroupoid $\G(\tG,\phi_{\tilde{H}}({\tilde H}))$ over $\phi_{\tilde H}({\tilde H})$ of $\G(\tG)$ and a subgroupoid $\G(G^*,H^\perp)$ over $H^\perp$ of $\G(G^*)$. They respectively integrate the subalgebroids $N^*\phi_{\tilde{H}}({\tilde H})\subset T^*\tG$ and $N^*H^\perp\subset T^*G^*$.

The following theorem establishes the role of coisotropic subgroups in the quotient of Poisson manifolds (the proof can be found in \cite{LuTh}).

\begin{thm}
\label{thm_poi_red}
Let $\sigma: K \times P \rightarrow P$ be a Poisson action of the Poisson Lie group $K$ over the Poisson manifold $P$ and let $B \subset K$ be a coisotropic subgroup of $K$. If the orbit space $B \backslash P$ is a smooth manifold, then there is a unique Poisson structure on $B \backslash P$ such that the natural projection map $P \rightarrow B \backslash P$ is Poisson.
\end{thm}

Let $H$ be a coisotropic subgroup of the Poisson Lie group $G$; if we apply this result to $P=K=G$, $B=H$ and to $P=K=G^*$, $B=H^\perp$ we conclude that both $H\backslash G$ and $G^*/H^\perp$ are Poisson manifolds.
Borrowing the terminology from quantum groups we call them {\it embeddable Poisson homogeneous spaces}, since $H$-invariant functions on $G$ are a Poisson subalgebra of $C^\infty(G)$. Let $\ph : G \rightarrow H \backslash G$ and $\php: G^*\rightarrow G^*/H^\perp$ the projection maps and let us denote $\ph(g)=\underline{g}$ and $\php(\gamma)=\underline{\gamma}$.
Embeddable Poisson homogeneous spaces come with a distinct point, the image $\underline{e}$ of the identity $e$, where the coinduced Poisson structure vanishes. Indeed, they can be characterized as those having at least one point where the Poisson structure vanishes, or equivalently the stability group of such a point is coisotropic.

\subsection{The complete and simply connected case}

Let us assume that $G$ is simply connected and complete. Let us review the concept of symplectic reduction via Lu's momentum map.

\begin{dfn} \label{dfn_momap}
 A $C^{\infty}$ map $J: P \rightarrow G^*$ is called a momentum mapping for the left Poisson action $\sigma: G\times P\rightarrow P $ if
 \[
\sigma_{\mbox{\tiny X}} ~ = ~ - \pi_{\mbox{\tiny P}}^\sharp
 (J^{*}(\cx)), ~~~ \forall X \in \la.
\]
where for each $X \in \la$, $\cx$ is the right invariant $\onf$ on $\Gs$ with value $X$ at $e$, and $\sigma_{\mbox{\tiny X}}$ is the fundamental vector field associated to $X$ by the action $\sigma$. The momentum mapping $J$ is said to be equivariant if $J(\sigma(g,p))={}^g(J(p))$, for any $g\in G$ and $p\in P$.
\end{dfn}

Remark that this definition is slightly different from the one given in \cite{LuTh}. If $P$ is symplectic, a Poisson action $\sigma$ admits a momentum mapping $J$ if and only if there exists a symplectic action of the symplectic groupoid $\G(G^*)$ on $P$ with anchor $J$: the correspondence is given by $\sigma(g,p)=(gJ(p))p$, for $g\in G$, $p\in P$, see \cite{WX}. Remark that this correspondence demands $G$ to be complete. By applying this result to $P=\G(G)$ we get that the left $G$-action on $\G(G)$ given by $\sigma(g,(k\gamma))=gk\gamma$ is Poisson and admits an equivariant momentum mapping $J(g\gamma)={}^g\gamma$. Moreover it is multiplicative, {\it i.e.} if $(g_1\gamma_1)$ and $(g_2\gamma_2)$ are composable then $J((g_1\gamma_1)(g_2\gamma_2))=J(g_1\gamma_1)J(g_2\gamma_2)$.
More concretely, for any $X\in\la$, the fundamental vector field of the left $G$-action is $\sigma_X(g\gamma)=r_{(g\gamma)^{-1}*}X$ and the right invariant form $X^r_\gamma=r^*_{\gamma^{-1}}X$. We have that, for any $g\gamma\in\G(G)$,
\begin{equation}\label{lem_momap_poi_ac}
\sigma_X(g\gamma)=-\sharpplus\circ [T_{g\gamma} J]^* (X^r_{J(g\gamma)})\;.
\end{equation}

\smallskip

Let us introduce the map
\begin{equation}\label{momentum_map}
\JH = \php \circ J,~~~~~  \JH(g\gamma)=\underline{{}^g\gamma}.
\end{equation}
Being a composition of Poisson submersions, $\JH$ is a Poisson submersion too. In the special case in which $H$ is a Poisson Lie subgroup, $H^\perp$ is a normal subgroup and $G^*/H^\perp\equiv H^*$ is a Poisson Lie group with Lie algebra $\h^\star=\la^\star/\h^\perp$. By using (\ref{properties_dressing}) it can be easily shown that the dressing action of $H$ descends to $G^*/H^\perp$, ${}^h\php(\gamma)\equiv\php({}^h\gamma)$, for any $h\in H$ and $\gamma\in G^*$. Then $\JH$ is an equivariant and multiplicative momentum mapping for the left multiplication by $H$. In the general coisotropic case, $G^*/H^\perp$ is only a Poisson manifold and $\JH$ must be thought as a momentum mapping in a generalized sense, see Corollary \ref{gpd_gpd_red} at the end of this subsection.

Let us consider $\JH^{-1} (\underline{e})=\{g\ga\in G \times G^*:{}^g \ga \in H^\perp\}$, that is a submanifold since $\JH$ is a submersion. Since $\{\underline{e}\}$ is a zero dimensional leaf of $G^*/H^\perp$ then  $\JH^{-1} (\underline{e})$ is coisotropic in $G \times G^*$. Let us show that the left multiplication by $H$ leaves $\JH^{-1} (\underline{e})$ invariant: Lemma \ref{lem_h_hperp_inacca} implies that, if $g\ga \in J_{\mbox{\tiny H}}^{-1} (\underline{e})$, that is if ${}^g \ga \in H^\perp$, then ${}^{h g} \ga \equiv {}^h ({}^g \ga) \in {}^h (H^\perp) = H^\perp$, and then $h g \ga \in J_{\mbox{\tiny H}}^{-1} (\underline{e})$. The left $H$ action is proper and free so that the orbit space is smooth.

\begin{thm} \label{thm_groupoid}
$\mathcal{G} (H\backslash G) = H\backslash \JH^{-1} (\underline{e})$ is a symplectic groupoid that integrates $H\backslash G$.
\end{thm}

\begin{proof} Since $J_{\mbox{\tiny H}}^{-1} (\underline{e})$ is a coisotropic submanifold of $\mathcal{G} (G)$, it admits a symplectic reduction and let $\omega_+^H$ the induced symplectic structure. Let us show that this reduction coincides with the quotient by the left $H$ action. Formula \ref{lem_momap_poi_ac} implies that the fundamental vector fields associated to the left multiplication of $G \times G^*$ by $H$ are given by
\begin{equation}\label{actionofH}
\sigma_{\mbox {\tiny X}} (g \ga) = -\pi_{\mbox{\tiny +}}^\sharp \left[ [T_{g \ga} J]^* \circ r_{J(g \ga)^{-1}}^* X \right], ~~~ \forall X \in \fh,
\end{equation}
where we recall that $J : \mathcal{G}(G) \rightarrow G^* : g\ga \mapsto {}^g \ga$ is the momentum mapping. On the other hand, the coisotropic distribution is defined as
\[
\pi_{\mbox{\tiny +}}^\sharp \left( N^*_{g \ga} J_{\mbox{\tiny H}}^{-1} (\underline{e}) \right),~~~ \forall g \ga \in J_{\mbox{\tiny H}}^{-1} (\underline{e}).
\]
Simple identities allow to write for $g\gamma\in\JH^{-1}(\underline{e})$
\begin{eqnarray*}
N^*_{g \ga} \JH^{-1} (\underline{e}) &=& {\rm Ker} [T_{g \ga} \JH]^\perp ={\rm Im}[T_{g\gamma} \JH]^* =\{[T_{g \ga} J]^* \circ [T_{J(g \ga)} p_{\mbox{\tiny H}^\perp} ]^* X, ~~~ \forall X \in \fh \} \cr
&=& \{[T_{g \ga} J]^* \circ r^*_{J(g\gamma)^{-1}}X, ~~~ \forall X \in \fh \},
\end{eqnarray*}
since $\php\circ r_{J(g\gamma)^{-1}}=\php$ and $[T_e\php]^*:T^*_{\underline e} (G^*/H^\perp)\equiv\fh\rightarrow T^*_eG^*\equiv\la$ is just the inclusion map.

Let us show that $J_{\mbox{\tiny H}}^{-1} (\underline{e})$ is a subgroupoid of $\mathcal{G}(G)$.
Take $g_1\ga_1,g_2\ga_2 \in J_{\mbox{\tiny H}}^{-1} (\underline{e})$ such that $g_2 = g_1^{\ga_1}$. Then $m(g_1\ga_1,g_2\ga_2) = g_1 \ga_1 \ga_2 \in J_{\mbox{\tiny H}}^{-1} (\underline{e})$ since ${}^{g_1} (\ga_1 \ga_2) = {}^{g_1} \ga_1 ~ {}^{g_1^{\ga_1}} \ga_2 \in H^\perp$ because ${}^{g_1^{\ga_1}} \ga_2 = {}^{g_2} \ga_2 \in H^\perp$ by definition. The quotient $H\backslash \JH^{-1} (\underline{e})$ inherits the structure of groupoid. In fact the left $H$ action on $\JH^{-1}(\underline{e})$ and on $G$ defines the relations $S_H\subset \JH^{-1}(\underline{e})\times \JH^{-1}(\underline{e})$ and $R_H\subset G\times G$, respectively; one can show that $(S_H,R_H)$ is a smooth congruence on $\JH^{-1}(\underline{e})$, according to Definition 2.4.5 of \cite{MK}, that induces a unique Lie groupoid structure on the quotient. We will follow a more direct way, by explicitly defining this groupoid structure. Let us denote with $\underline{g}\in H\backslash G$ the equivalence class of $g\in G$ and with $\underline{g}\gamma$ the equivalence class of $g\gamma\in \JH^{-1}(\underline{e})$. The source and target maps are defined as
\begin{eqnarray}
  \alphaGH (\underline{g} \ga) &=& \underline {\alpha_{\mbox{\tiny G}} (g \ga)} = \underline {g} \nonumber \\
  \betaGH (\underline{g}\ga) &=& \underline{\beta_{\mbox{\tiny G}}(g \ga)} = \underline {g^\ga} \;; \nonumber
 \end{eqnarray}
one must check that the definition of $\betaGH$ is correct; indeed $\betaGH (\underline{h g} \ga) = \ul{(h g)^\ga} = \ul{h^{{}^{g} \ga} ~ g^\ga}=\underline {g^\ga}$ since ${}^{g} \ga \in H^\perp$. Given $\underline{g_i} \ga_i \in H \backslash \JH^{-1}(\underline{e})$,
$i = 1, 2$, such as $\alphaGH (\underline{g_2} \ga_2) = \underline{g_2} = \underline {g_1^{\ga_1}} = \betaGH (\underline{g_1} \ga_1)$, we set $\mGH(\underline{g_1}  \ga_1 , \underline{g_2} \ga_2)=
\underline{g_1} \ga_1 \ga_2$. Then $\alphaGH ( \underline{g_1} \ga_1 \ga_2 ) = \underline{g_1} =
\alphaGH ( \underline{g_1} \ga_1 )$ and $\betaGH ( \underline{g_1}  \ga_1 \ga_2) =
\underline{g_1^{\ga_1\ga_2}} = \underline{(h g_2)^{\ga_2}} =\underline{h^{{}^{g_2} \ga_2} ~ {g_2^{\ga_2}}}=
\betaGH (\underline{g_2} \ga_2)$, where $h \in H$ is such that $h g_2 =
g_1^{\ga_1}$, which follows from the condition of composability, and last equality follows, once more, from the condition that ${}^{g_2}\gamma_2\in H^\perp$.

Finally, let us show that the reduced symplectic form $\omega_+^H$ is multiplicative. First, we observe that the restriction of $\omega_+=\pi_+^{-1}$ to $\JH^{-1}(\underline{e})$ is multiplicative, making it a presymplectic groupoid. Then we observe that the quotient map $\ph^{(1)}:\JH^{-1}(\underline{e})\rightarrow \G(H\backslash G)$ induces a submersion $\ph^{(2)}:\JH^{-1}(\underline{e})^{(2)}\rightarrow\G(H\backslash G)^{(2)}$, so that any element in $\Lambda^2 T\G(H\backslash G)^{(2)}$ can be written as $\ph_*^{(2)}V$  for $V\in \Lambda^2 T\JH^{-1}(\underline{e})^{(2)}$. We then have
$$
\langle (\mGH^*-\prGHuno^* - \prGHdue^*)\omega_+^H,\ph_*^{(2)} V\rangle = \langle (m^*-\pruno^*-\prdue^*)\omega_+,V \rangle = 0\;.
$$
\end{proof}
\smallskip

One can think of this reduction as a reduction of a symplectic groupoid action. Let $\G(G^*/H^\perp)$ the symplectic groupoid integrating $G^*/H^\perp$ obtained by the right counterpart of the above procedure. In total analogy with the above construction we have that $\G(G^*/H^\perp)=\{g\underline{\gamma}\in G^*\times G^*/H^\perp, \; g^\gamma\in H\}$. The groupoid structures are $\alphaGsHp(g\underline{\gamma})=\underline{{}^g\gamma}$, $\betaGsHp(g\underline{\gamma})=\underline{\gamma}$, $\mGsHp[(g_1\underline{\gamma_1})(g_2\underline{\gamma_2})]=g_1g_2\underline{\gamma_2}$ for $\underline{\gamma_1}=\underline{{}^{g_2}\gamma_2}$, etc... It is also clear that the isotropy group of $\underline{e}$ is $\G(G^*/H^\perp)^{\underline{e}}_{\underline{e}}=H$. One can easily check that $\JH:\G(G)\rightarrow G^*/H^\perp$ is the anchor for the symplectic action of $\G(G^*/H^\perp)$ on $\G(G)$ given by $(k\underline{\lambda})(g\gamma)=kg\gamma$.

\begin{cor}\label{gpd_gpd_red}
$\G(H\backslash G) = \G(G^*/H^\perp)^{\underline{e}}_{\underline{e}}\backslash\JH^{-1}(\underline{e})$\;.
\end{cor}

\subsection{The general case}\label{generalcase} If $G$ is neither complete nor simply connected, we have to use
the general form of $\G(G)$ given in Proposition \ref{sympl_grpd_non1conn}. Most of the construction of the
complete and simply connected case can be generalized in a straightforward way, apart from few crucial facts. Instead of Lu's
momentum map, we have to think of $\G(G)$ as a hamiltonian $\G(G^*)$ space, as described in Corollary \ref{grpd_action}. In
particular, the map $\JH$ defined as in (\ref{momentum_map}) is still a Poisson submersion and $\JH^{-1}(\underline{e})$ is a
coisotropic submanifold of $\G(G)$. The coisotropic reduction is therefore well defined. First of all, care has to be taken about smoothness of this quotient. In fact, in the general case the $\G(G^*)$-action of Corollary \ref{grpd_action} does not define a left $G$-action, due to non completeness of dressing vector fields. In particular, formula (\ref{actionofH}) still defines an infinitesimal action of $\mathfrak{h}$, spanning the coisotropic distribution, that cannot
in general be integrated to a group action of $H$. 

We first remark that it can be integrated to a groupoid action. In fact the restriction of the groupoid action $a:\G(G^*){}_\betaGs\!\times_J\G(G)\rightarrow \G(G)$ defined in Corollary \ref{grpd_action} to the subgroupoid $\G(G^*,H^\perp)$ defined in (\ref{subgroupoidofH}) leaves $\JH^{-1}(\underline{e})$ invariant, {\it i.e.} we have a left groupoid action  $a:\G(G^*,H^\perp){}_\betaGs\!\!\times_J\JH^{-1}(\underline{e})\rightarrow\JH^{-1}(\underline{e})$ with anchor $J:\JH^{-1}(\underline{e})\rightarrow H^\perp$. Furthermore, its infinitesimal action is the restriction of the algebroid action of $T^*G^*$ on $\G(G)$, and by repeating the argument in the proof of Theorem \ref{thm_groupoid}, spans the coisotropic distribution of $\JH^{-1}(\underline{e})$. The coisotropic quotient can be obtained as $\G(G^*,H^\perp)\backslash \JH^{-1}({\underline e})$, the orbit space of the groupoid action of $\G(G^*,H^\perp)$. Although $\JH^{-1}(\underline{e})$ is still a subgroupoid of $\G(G)$, it is not obvious a priori that this groupoid action defines a smooth congruence in $\JH^{-1}({\underline e})$ and so a groupoid structure on the quotient.

In order to overcome these problems, we introduce a weaker notion of completeness of dressing vector fields reducing the problem to a quotient by an ordinary free and proper group action. We say that two Lie algebras $(\mathfrak{g}_1,\mathfrak{g}_2)$
are a {\it matched pair of Lie algebras} if there exists a third Lie algebra ${\mathfrak g}_1\bowtie{\mathfrak g}_2$, called the {\it double Lie algebra}, isomorphic to ${\mathfrak g}_1\oplus{\mathfrak g}_2$ as vector space and containing ${\mathfrak g}_1$ and ${\mathfrak g}_2$ as Lie subalgebras. It comes out that $a$ and $b$ defined by $[X,\xi]=b_\xi(X)\oplus a_X(\xi)$ for $X\in {\mathfrak g}_1$, $\xi\in{\mathfrak g}_2 $, define compatible actions of the two Lie algebras on each other. 
Accordingly, we say that two Lie groups $G_1$ and $G_2$ form {\it a matched pair of Lie groups} if there exists a third Lie group $G_1\bowtie G_2$, called the {\it double Lie group}, diffeomorphic to $G_1\times G_2$ and containing $G_1$ and $G_2$ as closed Lie subgroups. It comes out that $G_1$ and $G_2$ act on each other with actions satisfying  compatibility conditions analogue to (\ref{properties_dressing}); viceversa, given such compatible actions then there exists a unique double Lie group (see \cite{Mj}, \cite{LW2}).
 
It is clear that the couples $({\mathfrak g},{\mathfrak g}^\star)$ and $(\mathfrak{h},\mathfrak{h}^\perp)$ form  matched pairs of Lie algebras with the coadjoint actions as compatible actions. We say that $(H,H^\perp)$ are
{\it relatively complete} if the infinitesimal actions of $\mathfrak{h}$ on $H^\perp$ and of $\mathfrak{h}^\perp$ on $H$ via
the dressing vector fields can be integrated in such a way that $(H,H^\perp)$ forms a matched pair of Lie groups, or equivalently that
the finite actions of $H$ and $H^\perp$ satisfy relations (\ref{properties_dressing}). We need the following Lemma concerning the universal cover of $H$, in order to include the case when $G$ is not simply connected.

\begin{lem}\label{lift_dressing_univ_cover}
\begin{itemize}
\item[$i$)] If $(H,H^\perp)$ forms a matched pair of Lie groups then also $(\tilde{H},H^\perp)$ forms a matched pair,
where $\tilde{H}$ is the universal cover of $H=\tilde{H}/Z_H$. 
 
\item[$ii$)] The center $Z_H\subset\tilde{H}$ acts trivially on $H^\perp$ and is a fixed point set of the dressing action of $H^\perp$.

\item[$iii$)] The quotient map $\tilde{H}\rightarrow H$ intertwines the $H^\perp$ action, {\it i.e.} $\tilde{h}^\gamma\rightarrow h^\gamma$ for any lift $\tilde{h}$ of $h\in H$ and $\gamma\in H^\perp$.

\end{itemize} 
\end{lem}
\begin{proof}
Let us prove point $i$). Any $\tilde{h}\in\tilde{H}$ can be seen as the equivalence class $[h]$ of a path $h:[0,1]\rightarrow H$, with $h(0)=e$,  with respect to homotopies preserving end points. Since $H^\perp$ is connected and its action on $H$ preserves the identity, for any $\gamma\in H^\perp$ and $[h]\in\tilde{H}$, define the action of $H^\perp$ on $\tilde{H}$ as $[h]^\gamma=[h^\gamma]$, where
$h^\gamma(t)=h(t)^\gamma$, and the action of $\tilde{H}$ on $H^\perp$ as ${}^{[h]}\gamma={}^{h(1)}\gamma$. It is immediate to
check that they are well defined and that ${}^{[h]}(\gamma_1\gamma_2)={}^{[h]}\gamma_1 {}^{[h]^{\gamma_1}}\gamma_2$. In order to check that $([h_1][h_2])^\gamma=[h_1]^{{}^{[h_2]}\gamma}[h_2]^\gamma$ we have to prove that the paths $t\rightarrow h_1(t)^{ {}^{h_2(t)}\gamma}$
and $t\rightarrow h_1(t)^{{}^{h_2(1)}\gamma}$ are homotopic. This can be shown by using the homotopy $F(s,t)=h_1(t)^{{}^{h_2(t+s(1-t))} \gamma}$. 

In order to get point $ii$), we observe that $Z_H=\pi_1(H)$ is realized as homotopy classes of loops and its action on $H^\perp$ is trivial; the action of $H^\perp$ on $Z_H$ is trivial since $H^\perp$ is connected to the identity so that any loop $z$ is homotopic to $z^\gamma$. Finally, point $iii$) can be directly verified once one realizes that the quotient $\tilde{H}\rightarrow H$ is realized as $[h]\rightarrow h(1)$.  
\end{proof}

Let $\tilde{H}\bowtie H^\perp$ be the double Lie group with the product rule given by $\tilde{h}\gamma={}^{\tilde{h}}\gamma \tilde{h}^\gamma$, for $\tilde{h}\in \tilde{H}$ and $\gamma\in H^\perp$.

\smallskip
\begin{thm}
\label{thm_groupoid_general} If $(H,H^\perp)$ are relatively complete then the left groupoid action  $a:\G(G^*,H^\perp)$ ${}_\betaGs\!\!\times_J\JH^{-1}(\underline{e})\rightarrow\JH^{-1}(\underline{e})$ is equivalent to a free and proper action of $H$ and  $H\backslash\JH^{-1}(\underline{e})$ is a symplectic groupoid that integrates $H\backslash G$.
\end{thm}
\begin{proof}
Let $\tG$ be the universal covering of $G$ so that $G=\tG/Z$ and let $\G(G)=Z\backslash\G(\tG)$ be as described in Proposition \ref{sympl_grpd_non1conn}. Let us recall that $\phi_1:\tG\rightarrow D$ and $\phi_2:G^*\rightarrow D$ are the Lie group homomorphisms entering the definition of $\G(\tG)$. 

We have that $[\tilde{g}_1,\gamma_1,\gamma_2,\tilde{g}_2]\in \JH^{-1}(\underline{e})$ if $\gamma_2\in H^\perp$. Let $\phi_{\tilde H}:\tilde{H}\rightarrow\tG$ be the Lie group homomorphism induced by the injection $\h\rightarrow {\mathfrak g}$. Then we have that $\phi_{\tilde{H}}(Z_H)\subset Z$.
Moreover, due to the uniqueness of the group homomorphism integrating any Lie algebra morphism, we can conclude that $\psi:\tilde{H}\bowtie H^\perp\rightarrow D$ defined as $\psi(\tilde{h}\gamma)=\phi_1(\phi_{\tilde{H}}(\tilde{h}))\phi_2(\gamma)$, for $\tilde{h}\in\tilde{H}$ and $\gamma\in H^\perp$, is a group homomorphism. For any $\tilde{h}\in\tilde{H}$, $\lambda\in H^\perp$ we have that $(\phi_{\tilde{H}}(\tilde{h}),\lambda,{}^{\tilde{h}}\lambda,\phi_{\tilde{H}}(\tilde{h}^\lambda))\in\G(G^*,H^\perp)$. In fact, by using the definition above of $\psi$ we have
\begin{eqnarray}\nonumber
\phi_1(\phi_{\tilde{H}}(\tilde{h}))\phi_2(\lambda)&=&\psi(\tilde{h}\lambda)=
\psi({}^{\tilde h}\lambda\tilde{h}^\lambda)=\psi({}^{\tilde h}\lambda)\psi(\tilde{h}^\lambda)=\cr& =&
\phi_2({}^{\tilde h}\lambda)\phi_1(\phi_{\tilde{H}}(\tilde{h}^{\lambda}))\;.
\end{eqnarray}

So we can define the left $H$-action on $\JH^{-1}(\underline{e})$ by choosing any lift $\tilde h$ of $h\in H$ and letting $$h[\tilde{g}_1,\gamma_1,\gamma_2,\tilde{g}_2]=a\{(\phi_{\tilde{H}}(\tilde{h}),\gamma_2,{}^{\tilde{h}}\gamma_2,\phi_{\tilde{H}}(\tilde{h}^\gamma_2))[\tilde{g}_1,\gamma_1,\gamma_2,\tilde{g}_2]\}=[\phi_{\tilde{H}}(\tilde{h})\tilde{g}_1,\gamma_1,{}^h\gamma_2,\phi_{\tilde{H}}(\tilde{h}^{\gamma_2})\tilde{g}_2]\;.$$ 

The independence on the choice of the lift $\tilde{h}$ is clear since $(\tilde{h}z)^\gamma=\tilde{h}^\gamma z$ for any $z\in Z_H$ due to point $ii$) in Lemma \ref{lift_dressing_univ_cover}.

Under this condition, the coisotropic reduction is obtained as a quotient of the free and proper action of $H$ and so it is a smooth manifold. Moreover, groupoid structures descend to the quotient, as it can be directly verified and everything goes through like in the proof of Theorem \ref{thm_groupoid}.
\end{proof}

\smallskip
In the following we analyze some obvious conditions that imply relative completeness.

\begin{lem}\label{relative_completeness}
If $H^\perp$ is simply connected and $H$ is a Poisson-Lie subgroup, then $(H,H^\perp)$ are relatively complete.
\end{lem}
\begin{proof} Due to Lemma \ref{lem_h_hperp_inacca}, the dressing vector fields of $H^\perp$ restricted to $H$ are zero and the action of $H^\perp$ is trivially integrated. Since the hypothesis $ii$) of Lemma 4.1 in \cite{Mj} is obviously satisfied and $H^\perp$ is simply connected, we get the result. \end{proof}

\subsection{An Example: $G=SU(1,1), H=U(1)$.}\label{example_su11}
Let us consider the following double Lie algebra ${\mathfrak d}=s\ell(2,{\mathbb C})$ with pairing $\langle A,B\rangle={\rm Im\,Tr}(AB)$ and
$${\mathfrak g}=\mathfrak{su}(1,1)=\left\{\left(\begin{array}{cc}ia & b\cr b^*&-ia \end{array}\right)\;, a\in{\mathbb R}, b\in{\mathbb C}\right\}~~,$$
$${\mathfrak g}^\star=\mathfrak{sb}(2,{\mathbb C})=\left\{\left(\begin{array}{cc}a & n\cr0&-a \end{array}\right)\;, a\in{\mathbb R}\;, n\in{\mathbb C}\right\}\;.$$
Since the group
$$G=SU(1,1)=\left\{\left(\begin{array}{cc} \alpha&\beta\cr \beta^*& \alpha^*\end{array}\right),\, |\alpha|^2-|\beta|^2=1 \right\}$$
is embedded in $D=SL(2,{\mathbb C})$, formulas (\ref{poisson_tensors}) define a multiplicative Poisson structure on $SU(1,1)$, even if it is not simply connected. The simply connected dual group is
$$G^*=SB(2,{\mathbb C})=\left\{\left(\begin{array}{cc} A&N\cr 0& A^{-1}\end{array}\right),\, A>0, N\in{\mathbb C}\right\}\;.$$
Let us choose as subgroup $H\subset G$ the diagonal $U(1)$, which
is a Poisson-Lie subgroup; then $H^\perp$ is the closed subgroup
of $G^*$ of strictly upper diagonal matrices
$$
H^\perp = \left\{ \left(\begin{array}{cc}1 & N\cr 0&1\end{array}\right)\, ,\, N\in {\mathbb C}\right\}~~~~. 
$$
The quotient
$U(1)\backslash SU(1,1)$ is homeomorphic to the open disk and its quantization has been studied in \cite{Kor}.

Dressing transformations are not complete. An easy way of looking at it is the following. Let $g\in SU(1,1)$ and $\xi\in\mathfrak{g}^\star$: the flux $g_t$ of the dressing vector field corresponding to $\xi$ is given locally by the solution of $g\exp t\xi=\gamma_t g_t$ with $\gamma_t\in G^*$. We see that for $t\in{\mathbb R}$ the equation
$$
\left( \begin{array}{cc}\alpha&\beta\cr\beta^*&\alpha^*\end{array}\right)\left(\begin{array}{cc}1&t\cr0&1\end{array}\right)=\left(\begin{array}{cc}A_t&N_t\cr0&A_t^{-1}\end{array}\right)
\left(\begin{array}{cc}\sigma_t & \tau_t\cr\tau_t^*&\sigma_t^* \end{array}\right)~~~$$
admits in general solutions only for $t<t_0$ (for instance take $\alpha,\beta\in{\mathbb R}$ with $\beta\not=0$). In particular the dressing action of $H^\perp$ on $G$ is not complete. On the contrary, one easily computes that
$$
\left( \begin{array}{cc}\alpha&0\cr0&\alpha^*\end{array}\right)\left(\begin{array}{cc}1&b\cr0&1\end{array}\right)=\left(\begin{array}{cc}1&\alpha^2 b\cr0&1\end{array}\right)
\left(\begin{array}{cc}\alpha & 0\cr 0& \alpha^* \end{array}\right)~~~,$$
so that $(H,H^\perp)$ are relatively complete. Then Theorem \ref{thm_groupoid_general} produces a symplectic groupoid for the quotient Poisson structure on the disc. A subfamily of the whole family of covariant Poisson discs given in \cite{Kles} can be described in a similar manner.

\medskip

\subsection{Comparison with the construction in \cite{Lu2007}} In \cite{Lu2007} the most general Poisson homogeneous spaces of Poisson Lie groups are considered. Drinfeld in \cite{Dr} showed that Poisson structures on $H\backslash G$, such that the right $G$ action is Poisson, are naturally associated to lagrangian subalgebras ${\mathfrak l}\subset{\mathfrak d}$. The case of $H$ coisotropic, considered in this paper, corresponds to ${\mathfrak l}={\mathfrak h}\oplus {\mathfrak h}^\perp$.   
Let us assume that $i$) $G$ is a closed subgroup of any $D$ integrating $\mathfrak d$ (even not simply connected); $ii$) $H=L_H\cap G$, where $L_H$ is the connected subgroup of $D$ integrating $\mathfrak l$; $iii$) the infinitesimal action of $\mathfrak l$ on $G$ is integrated to a finite action of $L_H$. Then a Poisson groupoid for any
Poisson homogeneous space, even non embeddable, is constructed. Moreover
conditions for the Poisson structure to be non degenerate are
given. 

If we restrict to the embeddable homogeneous spaces, that we consider in the present paper, and to the complete case, the
groupoid is described as $G\times_H H^\perp$, the fibred product with respect to the right $H$-action on $G\times H^\perp$ given by
$(g,\gamma)h=(gh,{}^{h^{-1}}\gamma)$. In this case the Poisson structure is non degenerate. We can describe our symplectic
groupoid $\G(H\backslash G)$ as a fibred product $H^\perp\times_H G$ with respect to the left action $h(\gamma,g)=({}^h\gamma,hg)$,
via the correspondence $(\underline{g}\gamma)\in\JH^{-1}(\underline{e})\rightarrow [{}^g\gamma,g]\in  H^\perp\times_H G$. It is then clear that the reduction procedure in \cite{Lu2007} coincides with the right version of our procedure.

In the non complete case the two constructions are different. The groupoid in \cite{Lu2007} is described as $\Gamma=G\times_H
L_H/H$, where $L_H\subset D$ is the connected subgroup integrating ${\mathfrak l}={\mathfrak h}\oplus {\mathfrak
h}^\perp\subset\mathfrak{d}$ and $L_H/H$ is the homogeneous space by right quotient of $H\subset L_H$. The hypothesis that the action of $\mathfrak{l}$ on $G$ can be integrated implies that the dressing vector fields
corresponding to $\mathfrak{h}^\perp$ are complete on $G$. The Poisson structure is not known to be symplectic in general.

In order to realize that this construction is, in general, different from ours, it is enough to look at the trivial case, where
$H=\{e\}$ and $H^\perp=G^*$. This case obviously satisfies the relative completeness requirement: the symplectic groupoid
described in Theorem \ref{thm_groupoid_general} is obviously the unreduced one $\G(G)$. The construction in \cite{Lu2007} gives
$\Gamma=G\times G^*$, which is not a groupoid. In fact, the requirement that the action of $\mathfrak{l}={\mathfrak g}^*$
integrates is equivalent to completeness.

A less trivial case is given by the example discussed in subsection \ref{example_su11}. In that case, $L_H= H\bowtie H^\perp$. We saw in
fact that the infinitesimal dressing action of $H^\perp$ is complete only when restricted to $H$,
where it is trivial, and is not complete on the whole $SU(1,1)$.

\bigskip\bigskip

\section{Concluding remarks}

In the complete case, the symplectic groupoid $\G(H\backslash G)$ described in the previous section has the source fibre isomorphic to $H^\perp$, so it will be the unique source simply connected groupoid integrating $H\backslash G$ only if $H^\perp$ is simply connected. Moreover, since we are interested in the problem of quantization a more explicit description will be necessary. In particular it is natural to ask when it is symplectomorphic to $T^*(H\backslash G)$ with the canonical symplectic form. This problem will be addressed elsewhere, while in this section we will sketch a construction of a diffeomorphism between the symplectic groupoid and the cotangent bundle, that can be considered a first step in this direction.

In the complete case, the symplectic groupoid $\G(H\backslash G)$ can be described as the fibre bundle $H^\perp \times_H G$,
associated with the principal bundle $G\rightarrow H\backslash G$ and the dressing action of $H$ on $H^\perp$. Since the cotangent
bundle is the bundle associated to the coadjoint action on $\mathfrak{h}^\perp$, let us suppose that there exists a
diffeomorphism $s_H:{\mathfrak h}^\perp\rightarrow H^\perp$ that intertwines the coadjoint action of $H$ with the dressing action,
{\it i.e.} $s_H(\Ad^*_h\xi)={}^hs_H(\xi)$, for any $h\in H$, $\xi\in{\mathfrak h}^\perp$. We then have a fibre bundle
isomorphism that we describe as follows. Let us consider any trivialization of the principal bundle $G\rightarrow H\backslash
G$ given by the local sections $g_i:U_i\rightarrow G$ and transition functions $h_{ij}: U_i\cap U_j\rightarrow H$, such that
$g_j(x)=h_{ji}(x)g_i(x)$ for any $x\in U_i\cap U_j$. Then there exist local diffeomorphisms $S_i: \G|_{U_i} \rightarrow T^*
H\backslash G|_{U_i}$ given by:
\begin{equation}
 S_i(x \ga)= \Ad^*_{g_i(x)^{-1}}(s_H^{-1}({}^{g_i(x)} \ga)) \in \Ad^*_{g_i(x)^{-1}}{\mathfrak h}^\perp=T^*_x(H\backslash G) \;.\nonumber
\end{equation}
Since $s_H$ intertwines coadjoint and dressing action of $H$, $S_i=S_j$ on $U_i\cap U_j$ so that a global diffeomorphism $S:\G(H\backslash G)\rightarrow T^*(H\backslash G)$ is defined. Since the source map, when transported to the cotangent bundle, coincides with the bundle projection, the symplectic structure cannot be the canonical one, unless the Poisson structure on $H\backslash G$ is trivial.

Let us briefly see a class of examples where to apply the above construction. When $H$ is a Poisson Lie group and $H^\perp$ is of exponential type, {\it i.e.} $H^\perp=\exp{\mathfrak h}^\perp$, we can choose $s_H=\exp$.  In fact, since $\mathfrak{h}^\perp$ in an ideal, the coadjoint action of $H$ on $H^\perp$ is a Lie algebra morphism: $Ad_X^* \left( [\xi, \eta] \right) = [Ad_X^*\xi, \eta] + [\xi,Ad_X^* \eta]$, for all $X  \in \fh, \xi,\eta \in \fh^\perp$. Then due to the uniqueness of the group automorphism that integrates the coadjoint action we conclude that ${}^h\exp\xi=\exp\Ad^*_h\xi$.

While we we were finishing this paper, it appeared on the net paper \cite{stefanini} that contains very close results. It is shown that the Poisson action of a complete Poisson Lie group $H$ on an integrable Poisson manifold $P$ can be lifted to a groupoid action of $\G(H^*)$ on $\G(P)$; this fact allows one to obtain the groupoid integrating $P/H$ by symplectic reduction. The result coincides with Theorem \ref{thm_groupoid} in our paper when we take $P=G$ as a Poisson Lie group and $H\subset G$ a Poisson Lie subgroup. It would be nice to extend the results of \cite{stefanini} to the most general coisotropic reduction described in Theorem \ref{thm_poi_red} in order to get our Theorem \ref{thm_groupoid} in full generality as a particular case of this reduction scheme.

\end{document}